\documentclass[a4paper,11pt,reqno]{amsart}

\usepackage[latin1]{inputenc}

\usepackage{amsmath,amsthm,latexsym,amscd,amssymb,MnSymbol}
\usepackage[all]{xy}
\usepackage{a4}
\usepackage{hyperref}
\usepackage{url}
\usepackage{fancyhdr}
\usepackage{amscd}
\usepackage[usenames,dvipsnames]{xcolor}
\numberwithin{equation}{section}

\newtheorem{prop}{Proposition}[section]
\newtheorem{lem}[prop]{Lemma}
\newtheorem{theorem}[prop]{Theorem}
\newtheorem{cor}[prop]{Corollary}

\theoremstyle{definition}
\newtheorem{dfn}[prop]{Definition}

\theoremstyle{remark}
\newtheorem{rem}[prop]{Remark}

\newcommand{\SF}{{\rm sf}}

\def\dim{\mathop{\rm dim}}

\def\Id{\mathop{\rm Id}}

\def\End{\mathop{\rm End}}
\def\Coker{\mathop{\rm Coker}}
\def\Hom{\mathop{\rm Hom}}\def\Ch{\mathop{\rm Ch}}\def\CS{{\rm CS}}

\def\ch{\mathop{\rm ch}}
\def\Ch{\mathop{\rm Ch}}
\def\id{\mathop{\rm id}}

\newcommand{\cK}{{\mathcal K}}
\newcommand{\Sp}{\mathrm{Sp}}
\newcommand{\cliff}{\mathrm{cliff}}

\def\phi{\varphi}
\def\ep{\varepsilon}
\def\a{\alpha}
\def\O{\Omega}
\def\s{\sigma}\def\t{\tau}
\def\G{\Gamma}

\def\R{\mathbin{\mathbb R}}\def\Z{\mathbin{\mathbb Z}}
\def\Q{\mathbin{\mathbb Q}}
\def\C{\mathbb{C}}

\def\cU{{\mathcal U}}
\def\cF{\mathcal F}\def\cN{{\mathcal N}}
\def\cV{\mathcal V}
\def\cE{\mathcal E}

\def\Mt{\tilde{M}}

\def\Tor{\mathop{\rm Tor}}

\def\ind{\mathop{\rm ind}}
\def\dim{\mathop{\rm dim}}\def\Ind{\mathop{\rm Ind}\,}

\newcommand\Di{D\kern-7pt/}

\let\oldmarginpar\marginpar
\renewcommand\marginpar[1]{\-\oldmarginpar[\raggedleft\footnotesize #1]%
{\raggedright\footnotesize #1}}

\begin{document}

\title{Flat bundles, von Neumann algebras and $K$-theory with $\R/\Z$-coefficients}

\author{Paolo Antonini, Sara Azzali and Georges Skandalis}

\thanks{This research took place at the Institut de Math\'ematiques de Jussieu, CNRS - University Paris Diderot.\\ Paolo Antonini is funded by the \emph{Projet ANR KInd}.\\Sara Azzali is funded by an \emph{INdAM-Cofund Fellowship}. \\
We thank these institutions for their support.}
\maketitle

{\footnotesize
%
\vskip 2pt   Institut de Math{\'e}matiques de Jussieu, (UMR 7586), Universit\'e Paris Diderot (Paris 7)

\vskip -1pt  UFR de Math\'ematiques, {\sc CP} {\bf 7012} - B\^atiment Sophie Germain 

\vskip-1pt  5 rue Thomas Mann, 75205 Paris CEDEX 13, France

\vskip-1pt  paolo.anton@gmail.com, azzali@math.jussieu.fr, skandalis@math.univ-paris-diderot.fr
}

\begin{abstract}
Let $M$ be a closed manifold and $\alpha : \pi_1(M)\to U_n$ a representation. We give a purely $K$-theoretic description of the associated  element $[\alpha]$ in the $K$-theory of $M$ with $\R/\Z$-coefficients. To that end, it is convenient to describe the $\R/\Z$-$K$-theory as a relative $K$-theory with respect to the inclusion of $\C$ in a finite von Neumann algebra $B$. We use the following fact: there is, associated with $\alpha$, a finite von Neumann algebra $B$ together with a flat bundle $\cE\to M$ with fibers $B$, such that $E_\a\otimes \cE$ is canonically isomorphic with $\C^n\otimes \cE$, where $E_\alpha$ denotes the flat bundle with fiber $\C^n$ associated with $\alpha$. 
We also discuss the spectral flow and rho type description of the pairing of the class 
$[\alpha]$ with the $K$-homology class of an elliptic selfadjoint (pseudo)-differential operator $D$ of order $1$.
\end{abstract}


\tableofcontents

\section{Introduction}

Secondary invariants of geometric elliptic operators, such as the rho invariant of a unitary representation $\alpha:\G\to U_n$ of the fundamental group, 
 gain stability properties only when reduced modulo $\mathbb{Z}$. Indeed, Atiyah, Patodi and Singer in the seminal papers \cite{APS2, APS3}  proved that the modulo $\mathbb Z$ class of the reduced rho invariant of an elliptic selfadjoint operator $D$ can be described as the pairing of the $K$-homology class $[D]$ 
 with a $K$-theory class $[\alpha]$ with  $\R/\Z$-coefficients associated with $\alpha$. This result is called the \emph{index theorem for flat bundles}. The construction of  $[\alpha]\in K^1(M, \R/\Z)$ is using cohomology, and is based on the fact that $K$-theory with real coefficients is isomorphic to $H^*(M;\R)$. 

The model of $K^*(X, \R/\Z)$ in \cite{APS3} is built of two addenda: a torsion part $K^*(X, \Q/\Z)=\varinjlim K^*(X, \Z/n!\Z)$, and a free part which is the image of $K^*(X, \R):=K^*(X)\otimes \R$; it relies on the functorial properties of ordinary $K$-theory. 

Atiyah, Patodi and Singer then suggested that a direct description can be given in terms of von Neumann algebras. This idea has been an inspiration for many authors. See in particular  \cite{DHK, DHK2, Hu, KP, Ba}.

Beyond the Atiyah--Patodi--Singer (APS) one, a number of models of the $\R/\Z$-K-theory can be found in the literature, each one has its own features and flavor. Karoubi's and Lott's models are based on Chern--Weil and Chern--Simons theory \cite{Ka, Lo}. Basu implemented the APS suggestion building a model with bundles of modules over von Neumann algebras \cite{Ba}. Apart from the model of Karoubi, the constructions in \cite{Lo,Ba} are based in a more or less explicit way on the notion of connection.

\medskip The purpose of this paper is to give a canonical construction of the $\R/\Z$-$K$-theory class associated with a flat bundle using operator algebraic tools, and to compute the pairing with $K$-homology as a Kasparov product. In particular, the models of $K$-theory with coefficients used here are purely operator theoretic.

Note that for a $C^*$-algebra $A$ in the bootstrap category, one can define the $K$-theory of $A$ with $\mathbb R$-coefficients as $K_*(A; \R):=K_*(A\otimes B)$, where $B$ is any  ${\rm II}_1$-factor: in fact, by the K\"unneth property, the  group $K_*(A\otimes B)$ is independent of $B$ up to a canonical isomorphism - and coincides with the APS model for $A$ commutative. 

The model for the $K$-theory with $\R/\Z$-coefficients that we use is simply the ordinary relative $K$-theory of the inclusion $A\hookrightarrow A\otimes B$ or, equivalently, the group $K_{*-1}(A\otimes C_{_{i_0}})$ where $C_{_{i_0}}$ is the mapping cone of the unital inclusion $i_0:\C\hookrightarrow B$ for any $\rm{II}_1$-factor $B$. This is immediately shown to be well defined for any $C^*$-algebra in the bootstrap category. Let us mention that  this mapping cone has been used recently by Deeley to construct a model of $K$-homology with $\R/\Z$-coefficients \cite{De}.

To show that all the models are equivalent under canonical isomorphisms, we use also operator algebraic definitions of the $K$-theory with coefficients in $\mathbb{Q},\  \mathbb{Z}/n \mathbb{Z}$ and $\Q/\Z$. Our models are direct noncommutative generalizations of the APS ones: we have the inclusion of $\C$ in $n\times n$-matrices - and inductive limits (UHF algebras), whereas the APS definition employs the $K$-theory of the mapping cones of maps of degree $n$ on the circle - and projective limits.

\smallskip Let $M$ be a closed manifold with fundamental group $\G$. If $E_\a$ is a flat vector bundle over $M$ with holonomy $\a\colon \G\to U(n)$, Atiyah, Patodi and Singer's class $[\a]_{APS}\in K^1(M;\R/\Z)$ is based on the idea that $E_\a$ defines a torsion element in the reduced $K$-theory of $M$, in fact there exists $k\in \mathbb N^*$ such that the sum of $k$ copies of $E_\alpha $ is trivial. The class of $E_\a$ in $K$-theory with $\Z/k\Z$-coefficients is then added to the transgression form corresponding to the two flat connections on the trivial bundle with fiber $\C^{nk}$ thus obtained.

\smallskip
Our main result is the direct description of the element $[\alpha]\in K^1(M;\R/\Z)$ associated with $\alpha$: this is given by the pair of bundles $E_\a,\C^n$ with the (almost) canonical isomorphism after tensoring with a ${\rm II}_1$-factor $B$. Our construction consists of the following points:
\begin{enumerate}
\item there is a canonical flat bundle $\mathcal{E}$ with fiber a $\rm{II}_1$-factor $B$ associated with a morphism $u:\G \to U(B)$ (where $U$ is the unitary group of $B$); 
\item there is a canonical isomorphism $E_\a\otimes \mathcal{E}\longrightarrow  \mathbb{C}^n\otimes \mathcal{E}$;
\item the bundle $\cE$ is trivial in $K$-theory and may actually be assumed to be trivial.
\end{enumerate}

The $\rm II_1$-factor $B$ can be taken to be $L^\infty(U_n)\rtimes \G$ - \emph{i.e.} the ${\rm II}_1$ von Neumann algebra Morita equivalent to the foliated flat bundle $\tilde M\times_\a U_n$ - and the bundle $\cE$ encodes the bundle of frames for $E_\alpha$.

The canonical isomorphism in (2) derives from the fact that a vector bundle becomes trivial when it is lifted to the bundle of frames.

The point (3) is based on the fact that a flat bundle with fibers a $\rm{II}_1$-factor is trivial in $K$-theory. This is easily seen applying Atiyah's  $L^2$-index theorem for covering spaces \cite{Ati} (and its generalizations \cite{Lu,Sc} by L\"uck and Schick to every trace on $C^*(\G)$  (\footnote{As a side remark, we note that this result remains true for any trace on $\ell^1(\G)$ by showing that Atiyah's method extends.})), together with the property that $K$-homology separates points of $K^*(M;\R)$. Conversely, Chern--Weil theory can be used to prove the same result, and this in turn gives a different proof of the $L^2$-index theorem and its generalizations.

\medskip We then establish the independence from all choices involved in the construction of $[\a]$, and further show - using a Chern--Simons transgression argument with coefficients in von Neumann algebras - that our element is the same as the one constructed by Atiyah, Patodi and Singer.

We finally pair the $K$-theory class $[\a]$ with a $K$-homology class $[D]\in K_*(M)$  represented by a first order elliptic operator. The pairing is realized as a Kasparov product
\begin{equation*}
  KK^0(\C, C(M)\otimes C_{i_0})\times KK^1(C(M),\C)\longrightarrow KK^1(\C, C_{i_0})=K_1(C_{i_0})=  \R/\Z.
\end{equation*}

To compute the intersection product $[\alpha]\otimes [D]$, it is convenient to look at $[\alpha]$ in $KK(\mathbb{C},Z_{i_0}\otimes C(M))$, the $K$-theory of the double cylinder of the inclusion $C(M)\hookrightarrow C(M)\otimes B$. This provides a path which interpolates between the operator $D^\alpha $ obtained as a twisted tensor product of $D$ by the flat bundle $E_\a$, and the operator $D^n$ obtained by tensoring $D$ by the trivial bundle 
of rank $n=\rm{rk}(E_\a)$. The $KK$-product puts $D^\alpha $ and $D^n$ as affiliated to the same  ${\rm II}_\infty$ factor with  (relatively) compact resolvent, that are bounded perturbations of each other since they have the same principal symbol. Furthermore, they have discrete spectrum and the corresponding spectral projections have integer von Neumann dimension since they are obtained by tensoring with $B$ the Dirac operators $D^\alpha$ and $D^n$.
 
The pairing, as an element of $\R/\Z$, is given by a type ${\rm II}$ spectral flow: for every $a \in \R$, there is a well defined index between the two projections $\chi_{[a,\infty[}(D^\alpha)$ and $\chi_{[a,\infty[}(D^n)$ whose difference is relatively compact; this index does not depend on $a$ up to $\Z$.
 
Finally, since the spectral flow of Dirac type operators is related through variational formulas to the eta invariant (see \cite{BF, BGV, MZ2, CP1}), we get the index theorem for flat bundles
$[\alpha]\otimes [D]=\xi(D^\alpha)-\xi(D^n)\in \mathbb{R}/\mathbb{Z}.$


\section{Preliminary constructions}
We start with some well known constructions for $C^{*}$-algebras that we shall use in the following.

\subsection{Full modules}

Let $A$ be a unital $C^*$-algebra (or just a ring). 
 A (right) $A$-module $F$ is said to be \emph{full} if $ \text{span} \{\ell(x),\ x\in F,\ \ell\in \mathcal {L}(F, A)\}=A$ (where $\mathcal L(F, A)\}$ denotes the space of $A$-linear maps from $F$ to $A$). If $A$ is a non unital $C^*$-algebra, $F$ is full if $\text{span} \{\ell(x),\ x\in F,\ \ell\in \mathcal L(F, A)\}$ is dense in $A$.

\begin{lem}\label{full1}
Let $A$ be unital and $F$ be a finitely generated projective full module over $A$. Let $E, E'$ be finitely generated projective modules over $A$.  Then $[E]=[E']$ in $K_0(A)$ if and only if $\exists k\in \mathbb{N} : E\oplus F^k\simeq E'\oplus F^k$.
\end{lem}

\begin{proof}  
By hypothesis there exists a finitely generated projective $G$ such that $E\oplus G\simeq E'\oplus G$. Write $1=\sum \ell_i(x_i)$ with $\ell_1,\dots \ell_n\in F^*$ and $x_i\in F$; thus we construct the module map $f\colon G^n\to A\;\;, (y_1,\ldots ,y_n)\mapsto \sum _i\ell_i (y_i)$ which is onto. Since $G$ is finitely generated, there is a onto module map $g:A^m\to G$;  we deduce a surjective module map $h\colon F^{nm}\to G$.

Then it also holds $F^{nm}\simeq G\oplus G'$ (projectivity of $G$), and then $E\oplus F^{nm}\simeq E'\oplus F^{nm}$.

The converse is obvious.
\end{proof}

\begin{lem}\label{full2}
Let $A$ be a unital $C^*$-algebra, and $E_1$ and $E_2$ be finitely generated projective and full $A$-modules.
Suppose $[E_1]=[E_2]$ in $K_0(A)$. Then there exists $n$ such that $E_1^n\simeq E_2^n$.
 \end{lem}
 
\begin{proof} 
Since $E_1$ and $E_2$ are full, there exists $k$ such that $E_1^k\oplus E_1\simeq E_1^k\oplus E_2$ and $E_1\oplus E_2^k\simeq E_2\oplus E_2^k$.
   
Then $E_1^{a+1}\oplus E_2^b\simeq E_1^{a}\oplus E_2^{b+1}$, as soon as $a\ge k$ or $b\ge k$. We  obtain  $E_1^{2k}\simeq E_2^{2k}$. 
  \end{proof}

\subsection{Mapping cones}

The mapping cone of a morphism $\phi\colon A\to B$ of (unital) $C^*$-algebras, 
is the algebra $C_\phi=\{(a, \gamma)  \in A\oplus C_0((0,1], B) \;|\;  \gamma(1)=\phi(a)\}$.

Let now $A_0, A_1, B$ be unital $C^{*}$-algebras.
\begin{dfn}
Given two morphisms $\phi_i\colon A_i\to B$, $i=0,1$, define the double cylinder
algebra as
\begin{equation}
\label{dcylinder}
Z_{\phi_0, \phi_1}=\{(a_0, \s,a_1)  \text{ with }   a_i\in A_i, \s \in C ([0,1], B) \text{ s.t. }  \s(i)=\phi_i(a_i), i=0,1\}
\end{equation}
When $\phi_0=\phi_1=\phi$ we will denote the double cylinder $Z_{\phi_0, \phi_1}$ by  $Z_\phi$.

Note that for $A_0=0$ this construction is the cone $C_{\phi_1}$. When $A_0=B$ and $\phi_0$ is the identity of $B$, the double cylinder is sometimes called  the mapping cylinder of $\varphi_1$.
\end{dfn}
For the double cylinder $Z_\phi$ there is a split exact sequence
\begin{equation}
  \label{split}
  \xymatrix{ 0 \ar[r] &C_\phi \ar[r]_r &Z_\phi \ar[r]_{\pi_1}& A\ar@/_/[l]	_s \ar[r] & 0}
  \end{equation}
with $r(a, \s)=(0,\s, a)$ and $\pi_1(a_0, \s,a_1)=a_1$, and the splitting given by $s(a)=(a,\underline{\phi({a})}, a)$
so that $K_*(Z_\phi)\simeq K_*(C_\phi)\oplus K_*(A)$. In particular, there is a map from $K_0(Z_\phi)$ to the summand $K_0(C_\phi)$ of the form $x\mapsto x- (s\circ \pi_1)_*x$.

Note that if $\phi$ is injective, then  $$C_\phi=\{\gamma \in C_0((0,1], B) \;|\;  \gamma(1)\in \phi(A)\}\ \ \hbox{and}\ \ Z_\phi=\{\gamma  \in C([0,1], B) \;|\;  \gamma(0),\gamma(1)\in \phi(A)\}.$$

If $B$ is a ${\rm II}_1$-factor,  the inclusion $i_0\colon \C\to B$ gives  $K_0(C_{i_0})=0$ and $ K_1(Z_{i_0})=K_1(C_{i_0})=\R/\Z$.

\subsection{Relative $K$-theory}
\label{relKth}
Although relative $K$-theory coincides with the $K$-theory of a mapping cone, it helps in giving a more direct description of some $K$-theory elements. For instance, it has been recently used by Deeley to construct a model of $K$-homology with $\R/\Z$-coefficients \cite[Prop. 2.2,  Ex. 5.3]{De}.

 Let $\varphi : A \to B$ be a homomorphism of unital $C^{*}$-algebras. 
The group $K_{0}(\varphi)$ is given by generators and relations: 
\begin{itemize}
\item Its generators are triples $(E_{+},E_{-},u)$ where $E_{+}, E_{-} $ are finitely generated projective $A$-modules,  and $u \colon E_+\otimes_A B\to E_-\otimes_A B$ is an isomorphism.

\item Addition is given by direct sums.

\item A homotopy is a triple associated with the map $C([0,1];A)\to C([0,1];B)$ induced by $\varphi$.

\item Trivial elements are  those triples $(E_{+},E_{-},u)$ for which $u = \varphi(v)$ for some isomorphism $v\colon E_+\to E_-$.  The group $K_0(\varphi)$ is formed as the set of those triples divided homotopy and addition of trivial triples.

\item For non unital algebras, we just put $K_{0}(\varphi)=K_{0}(\tilde \varphi)$, where $\tilde \varphi:\tilde A\to \tilde B$ is obtained by adjoining units everywhere.

\item The group $K_1(\varphi)$ is also formed in this way. Its generators are pairs $(u,f)$ where $u$ is a unitary in $M_n(A)$ and $f$ is a continuous path in the unitaries of $M_n(B)$ joining $\varphi(u)$ to $1_n$.

\item There is a natural isomorphism $K_*(\varphi)\simeq K_*(C_\varphi)$. This morphism is almost tautological for $K_1$. For $K_0$, it is described as follows.  Assume that $\varphi$ is unital and use the isomorphism $K_*(Z_\varphi)=K_*(C_\varphi)\oplus K_*(A)$ associated the split exact sequence \eqref{split}. 
The class of a triple $(E_{+},E_{-},U)\in K_0(\varphi)$ is then the image in the summand $K_0(C_\varphi)$ of the $Z_\varphi$-module 
$$F_U=\{(x_+;f;x_-)\in E_+\times C\left([0,1];E_+\otimes_A B\right)\times E_-;\ f(0)=x_+\otimes 1 \ \hbox{and}\ Uf(1)=x_-\otimes 1\}.
$$

 Let $e^+$, $e^- \in M_n(A)$ be projections such that $E_+=e^+ A^n$, $E_-=e^- A^n$ and $u\in M_n(B)$ such that $u^*u=\phi (e^+)$, $uu^*=\phi (e^-)$ and satisfies $ux=U(x)$ for $x\in e_+A^n=E_+$. Let  $f,g\in C([0,1])$ be given by $f(t)=\cos \frac {t\pi}2$ and $g(t)=\sin \frac {t\pi}2$. Then we can write $F_U=eZ_\varphi^{2n}$ where $$e=\left(\begin{pmatrix} e^+&0\\0&0\end{pmatrix}, \begin{pmatrix} f^2\varphi(e^+)&fg u^*\\fg u&g^2 \varphi(e^-)\end{pmatrix},\begin{pmatrix} 0&0\\0&e^-\end{pmatrix}
\right)\in M_{2n}(Z_\varphi).$$
\end{itemize}

\subsection{The bootstrap category}

The bootstrap category is the smallest class $\cN$ of separable nuclear $C^*$-algebras containing commutative ones and closed under $KK$-equivalence \cite[22.3.4]{Bl}.

Every $C^*$-algebra $A$ in the bootstrap category satisfies the K\"unneth formula for tensor products \cite[Theorem 23.1.3]{Bl}; i.e. for every  $C^*$-algebra $B$
\begin{equation}
\label{Kuenneth}
 \xymatrix{
 0\ar[r]^{}& K_{\bullet}(A)\otimes K_{\bullet}(B) \ar[r]^-{\a} &K_{\bullet}(A\otimes B)\ar[r]^-{\s}& \Tor ^{\Z}_{1}(K_{\bullet}(A), K_{\bullet}(B))\ar[r]^{}& 0\\
  }
\end{equation}
where we denote $K_{\bullet}=K_0\oplus K_1$ as a graded group. Here $\a$ has degree $0$, and $\s$ has degree $1$. If one of $K_{\bullet}(A)$ and $K_{\bullet}(B)$ is torsion free, then $\a$ is an isomorphism.
\begin{rem}
If $A$ is in the bootstrap category and $B$ is a von Neumann algebra, then $K_1(B)=0$ and $K_0(B)$ is torsion free, see section \ref{KvN}. Therefore $K_*(A)\otimes K_0(B) \simeq K_*(A\otimes B)$.
\end{rem}

\subsection{$K$-theory of von Neumann algebras}
\label{KvN}
Any von Neumann algebra $A$ is uniquely decomposed in a product
$$A=A_{{\rm I}_f}\times A_{{\rm I}_{\infty}}\times A_{{\rm II}_1}\times A_{{\rm II}_{\infty}}\times A_{{\rm III}},$$
 consequently the $K$-groups uniquely split
$$K_{*}(A)=K_{*}(A_{{\rm I}_f})\times K_{*}( A_{{\rm I}_{\infty}})\times K_{*}( A_{{\rm II}_1})\times K_{*}(A_{{\rm II}_{\infty}})\times K_{*}(A_{{\rm III}}).$$  It is easy to see that if $A$ is properly infinite then $K_0(A)=0$ while it is always the case that $K_1(A)=0$. It follows that the $K$ theory is reduced to that of the finite piece
 $$K_0(A)=K_0(A_{{\rm I}_f})\oplus K_0(A_{{\rm II}_{1}}).$$
 
 If $A$ is finite, there is a unique center valued trace $\operatorname{tr}_A^u:A\longrightarrow Z(A)$ with the property that two projections $p,q\in M_n(A)$ are equivalent if and only if $\operatorname{tr}_A^u(p)=\operatorname{tr}_A^u(q)$. Then it follows from the universal property of $K_0$ that $\operatorname{tr}_A^u$ defined on projections extends to a \emph{von Neumann center valued dimension} which is an injection
 $$\operatorname{dim}^u_A:K_0(A)\longrightarrow Z(A)_{sa}:=\{a\in Z(A), \, a=a^*\}.$$ 
 If $A$ is type ${\rm II}_1$, this is an isomorphism.


\section{$K$-theory with coefficients}

A model for the $K$-theory of a $C^*$-algebra $A$ with coefficients in a countable abelian group $\Lambda$ is $K_*(A;\Lambda):=K_*(A\otimes B_\Lambda)$ where $B_\Lambda$ is a $C^*$-algebra in the bootstrap category with a specified isomorphism $K_0(B_\Lambda)\simeq \Lambda$ and such that $K_1(B_\Lambda)=0$. If $\Lambda$ is uncountable, such a model can be provided using (uncountable) inductive limits.

In this section, we describe $K$-theory with $\R/\Z$-coefficients in terms of von Neumann algebras. The two models coincide for $C^*$-algebras in the bootstrap category. This von Neumann description is suitable for the construction of the $K$-theory class of a flat bundle given in section \ref{section5}.

In order to relate our model with the one of Atiyah, Patodi and Singer \cite[Sec. 5]{APS3}, we briefly discuss models for $K$-theory with coefficients in $\Q,\ \R,\ \Z/n\Z$ and $\Q/\Z$ and compare them with the respective versions in \cite{APS2}.

\subsection{$K$-theory with rational and real coefficients}


For a (locally) compact space $X$, Atiyah--Patodi--Singer's model is  $K^*_{APS}(X, \mathbb Q):=K^*(X)\otimes \Q$ and $K^*_{APS}(X, \R):=K^*(X)\otimes \R$. More generally, we may set $K_*^{APS}(A, \mathbb Q):=K_*(A)\otimes \Q$ and $K_*^{APS}(A, \R):=K_*(A)\otimes \R$ for any $C^*$-algebra $A$.

Our description of $K$-theory with rational coefficients is the following.

\begin{dfn}
Let $D$ be the universal UHF algebra, \textit{i.e.} the one which satisfies $K_0(D)=\Q$. It is  obtained as an inductive limit of matrices, $D=\displaystyle\bigotimes_{n\in \mathbb{N}}M_n(\C)$.
Define

\begin{equation}
\label{K.Q}
K_*(A, \mathbb Q):=K_{*}(A\otimes D)
\end{equation}
\end{dfn}

Since $D$ is in the bootstrap category, $K_1(D)=0$ and $K_0(D)=\Q$, the K\"unneth formula gives a natural isomorphism $K_*(A)\otimes \Q\to K_*(A\otimes D)$, and therefore our $K$-theory with rational coefficients coincides with $K_*^{APS}(A, \mathbb Q)$.

In order to describe $K$-theory with real coefficients, we use the following result:

\begin{lem}
Let $B$ be a ${\rm II}_1$-factor, and  $A$ be a $C^*$-algebra in the bootstrap category. Then the group $K_*(A\otimes B)$ is canonically isomorphic to $K_*(A)\otimes \R$, and therefore it does not depend on $B$ up to a canonical isomorphism.
\end{lem}
\begin{proof}
The K\"unneth isomorphism
\begin{equation}
\label{Kunnet}
 \xymatrix{
 K_0(A)\otimes \R \ar[r]^-{\sim} &K_0(A\otimes B)\\
  }
\end{equation} is given by the canonical map $[x]\otimes t\mapsto [x\otimes p_t]$.\end{proof}

In particular, if $B_1$ is a ${\rm II}_1$-subfactor of $B_2$, the induced map $K_*(A\otimes B_1)\longrightarrow K_*(A\otimes B_2)$  is a canonical isomorphism, independent of the inclusion $B_1\subset B_2$.
In the following, $B$ will always denote a ${\rm II}_1$-factor. Then for $A$ in the bootstrap category we can define, independently of $B$ a model for $K$-theory with coefficients in $\R$, which is naturally isomorphic with the APS one.
\begin{dfn} Let $A$ be in the bootstrap category.
  \label{K*R}
    \begin{equation}\label{K.R}
K_*(A, \R):=K_*(A\otimes B)
    \end{equation}
\end{dfn}

\subsection{$K$-theory with $\Z/n\Z$-coefficients}

Let $i_n\colon \C\hookrightarrow M_n(\C)$ be defined by $i_n(1)=\Id$.

\begin{dfn}
  \label{K*Z/nZ}
    \begin{equation}\label{K.Z/nZ}
K_*(A, \Z/n\Z):=K_{*-1}(A\otimes C_{i_n})
    \end{equation}
\end{dfn}

Let us relate it with the $K$-theory with $\Z/n\Z$-coefficients as defined by Atiyah--Patodi--Singer in \cite[Sec. 5]{APS2}.

Let $f_n:S^{1}\to S^1$ given by $f_n(z)=z^n$. The APS definition of the $K$-theory with $\Z/n\Z$-coefficients for a closed (locally) compact space $X$ is the relative $K$-theory of $f_n\times \id_X:S^1\times X\to S^1\times X$, \cite[(5.2)]{APS2}.

Define $f_n^{*}\colon C(S^{1})\to C(S^{1})$ by $f_n^{*}(\phi)(z)=(\phi\circ f_n)(z)=\phi(z^{n})$.

The $K$-theory of the cone algebras $C_{i_n}$ and $C_{f_n^{*}}$ are computed by means of the exact sequences of the cones and give 
$K_1(C_{i_n})\simeq K_{0}(C_{f_n^{*}})\simeq \Z/n\Z$,  and $K_0(C_{i_n})\simeq K_{1}(C_{f_n^{*}})=0$.

Note that the inclusion $SM_n(\C)\to C_{i_n}$ induces the quotient map $\Z\to \Z/n\Z$ in $K_1$, and therefore the isomorphism $K_1(C_{i_n})\simeq  \Z/n\Z$ is canonical; in the same way, the isomorphism $K_{0}(C_{f_n^{*}})\simeq \Z/n\Z$ is canonical, and there is therefore a canonical isomorphism $\psi:K_1(C_{i_n})\to K_0(C_{f_n^{*}})$. By the UCT  \cite[Prop. 23.10.1]{Bl}, $KK^{1}(C_{i_n},C_{f_n^{*}})$ is isomorphic to $\Hom(\Z/n\Z,\Z/n\Z)$. The isomorphism $\psi$ determines therefore a generator of $KK^{1}(C_{i_n},C_{f_n^{*}})$. 

It follows that $C_{i_n}$ and $C_{f_n^{*}}$ are naturally $KK^1$-equivalent, or, equivalently, the algebras $SC_{i_n}$ and $C_{f_n^{*}}$ are canonically $KK$-equivalent.

It is of course possible to give an explicit construction of the canonical element of $KK^{1}(C_{i_n},C_{f_n^{*}})$.

As a consequence, let $A$ be any separable $C^*$-algebra. Denote now with $i_{n, A}\colon A\hookrightarrow A\otimes M_n(\C)$, $f_{n, A}^{*}\colon C(S^{1}, A)\to C(S^{1}, A)$ the analogous maps as above. Then $C_{i_{n,A}}=A\otimes  C_{i_n}$ and $C_{f_{n,A}^*}=C_{f_n^{*}}\otimes A$ are canonically $KK^{1}$-equivalent.

In particular, for a (locally) compact space $X$,
$
K^*_{APS}(X, \Z/ n\Z)\simeq K_{*+1}( C_0(X)\otimes C_{i_{n}})\ .
$

In other words, the definition \ref{K.Z/nZ} of $K$-theory of $\Z/n\Z$ coefficients coincides (for abelian $C^*$-algebras) with the one given in \cite{APS2}.

\subsection{$K$-theory with $\Q/\Z$-coefficients}

\begin{dfn}
Let $A$ be in the bootstrap category. Let $D$ be  the universal UHF algebra, and $A\hookrightarrow A\otimes D$ the obvious inclusion,  define
\begin{equation}
\label{K.QZ}
K_*(A, \mathbb Q/\Z):=K_{*-1}(C_{A\hookrightarrow A\otimes D})
\end{equation}
\end{dfn}

\begin{rem} For a (locally) compact space $X$, $K^*_{APS}(X, \mathbb Q/\Z)$ is defined as the inductive limit of $K^*_{APS}(X, \Z/ n!\Z)$ \cite[(5.3)]{APS2}. More generally, for every $C^*$-algebra $A$, we may set $K_*^{APS}(A, \mathbb Q/\Z):=\varinjlim_n K_*^{APS}(A, \Z/ n!\Z)$. Observe that the construction in APS can be made using the projective limit. If
\begin{equation*}
\label{telesc}
 \xymatrix{
T=\varprojlim \, ....\ S^1\ar[r]^-{\times n} &S^1\ar[r]^-{\times (n+1)} &S^1\ ... \\
  }
\end{equation*}
then $K_*^{APS}(A, \mathbb Q/\Z)$ is the relative $K$-theory of $A\otimes C(S^1)\to A\otimes C(T)$. Note also that $C(T)=C^*(\Q)$.

We have 
  $C_{A\hookrightarrow A\otimes D}=\varinjlim_n  C_{A\hookrightarrow A\otimes M_{n!}(\C)}$. As $K_*(A, \Z/n\Z)\simeq K_*^{APS}(A, \Z/n\Z)$, using continuity of $K$-theory, we find $K_*(A, \mathbb Q/\Z)\simeq K_*^{APS}(A, \mathbb Q/\Z)$.
\end{rem}

\subsection{$K$-theory with $\R/\Z$-coefficients}
\begin{dfn}
\label{K*RZ}
We propose the following realization of $\R/\mathbb{Z}$-$K$-theory of a $C^{*}$-algebra $A$ in the bootstrap category.
$$K_*(A,\R/\Z):=K_{*+1}(\operatorname{Cone}(A\hookrightarrow A\otimes B)),$$ 
where $B$ is any ${\rm II}_1$-factor. Note that Definitions \ref{K*R} and \ref{K*RZ} are  well posed: first of all since $A$ is nuclear there is no ambiguity in the $C^*$-tensor product $A\otimes B$. Furthermore the right hand side is in both definitions independent, up to natural canonical isomorphism, of the factor $B$. 
\end{dfn}
The latter definition is indeed a realization of $\R/\mathbb{Z}$-$K$-theory. 
The Bochstein change of coefficients is the long exact sequence 
associated to the mapping cone $C^*$-exact sequence. 

Atiyah, Patodi and Singer's description of $\R/\mathbb{Z}$-$K$-theory goes back to an idea of Segal \cite[Sec. 5]{APS3}. It is made up of two addenda. The first component is the torsion part, contained in $K^*(X,\mathbb{Q}/\mathbb{Z})$. The free part is in $K^*(X,\R)=K^*(X)\otimes \R$. More precisely
$$K_{APS}^*(X,\R/\mathbb{Z}):=\operatorname{cokernel}\{(p,-j):K^*(X,\mathbb{Q})\longrightarrow K^*(X,\mathbb{Q}/\mathbb{Z})\oplus K^*(X,\R)\}$$ where $p$ is the natural projection and $j$ is the natural injection.

We may of course put $$K^{APS}_*(A,\R/\mathbb{Z}):=\operatorname{cokernel}\{(p,-j):K_*(A,\mathbb{Q})\longrightarrow K_*(A,\mathbb{Q}/\mathbb{Z})\oplus K_*(A,\R)\}$$

Finally we prove: 
\begin{prop}
\label{RZvsAPS}
For an algebra $A$ in the bootstrap category, our realization of $K_*(A, \R/\Z)$ coincides with $K^{APS}_*(A,\R/\mathbb{Z})$.
\end{prop}
\begin{proof}
 Let $D$ be the universal UHF algebra, $B$ a ${\rm II}_1$-factor, and denote with $i_D\colon \C\hookrightarrow D$ and $i_B\colon \C\hookrightarrow B$ the obvious inclusions. It is enough to prove that
\begin{equation}
  \label{realizK}
  K_{*}(A\otimes C_{i_B})=\Coker \left\{(p, -j) \colon K_{*+1}(A\otimes D)\longrightarrow  K_{*}(A\otimes C_{i_D})\oplus K_{*+1}(A\otimes B)\right\}
\end{equation}
where $p$ is induced from the natural map $(C(M)\otimes D)(]0,1[)\longrightarrow C_{i_D}$ and $j$ is induced from any unital inclusion $D\hookrightarrow B$.
To verify \eqref{realizK}, consider the diagram 
$$
\xymatrix{A\otimes V_D\ar[r]^p\ar[d]_j& A\otimes C_{i_D}\ar[d]\\
A\otimes V_B\ar[r]_h&A\otimes C_{i_B} }$$
where $V_B=\left\{(a,s)\in C_0(]-1,0])\times B([0, 1[) \text{ s.t. }  i_B(a(0))=r(0)\right\}$ and $V_D$ is defined in analogous way. Note that the inclusion $A\otimes SB\to A\otimes V_B$ induces an isomorphism in $K$-theory as follows from the exact sequence $0\to (A\otimes B)(]0,1[)\to V_B\to A(]{-}1,0])\to 0$. Moreover $h$ is surjective. Therefore, we have an associated Mayer-Vietoris exact sequence. Now $j_*:K_*(A\otimes V_D)\simeq K_{*+1}(A\otimes D)\longrightarrow  K_{*+1}(A\otimes B)\simeq K_*(A\otimes V_D)$ is injective.

Finally the Mayer--Vietoris sequence gives 
$$
0\longrightarrow K_*(A\otimes D)\longrightarrow  K_{*+1}(A\otimes C_{i_D})\oplus K_*(A\otimes B)\longrightarrow K_{*+1}(A\otimes C_{i_B})\longrightarrow 0.\qedhere
$$
\end{proof}

\begin{rem}
There are in the literature several other realizations of $\R/\mathbb{Z}$-$K$-theory:
\begin{description}
\item[Karoubi's desription] is a cohomological interpretation based on Chern--Weil and Chern--Simons theory \cite{Ka2} - this will be explained below. At first sight, it works only for manifolds but as pointed out by Atiyah Patodi and Singer it can be generalized to compact spaces by an embedding trick. Cycles for $K^1(X,\R/\mathbb{Z})$ are triples $((E,\nabla^E),(F,\nabla^F),\omega)$ where $E,F$ are Hermitian bundles with Hermitian connections $\nabla^E,\nabla^F$ and $\omega$ is a transgression odd degree differential form such that
$$d\omega=\operatorname{Ch}(\nabla^E)-\operatorname{Ch}(\nabla^F).$$
A suitable notion of sum and equivalence of cycles is defined leading to an abelian group which is  a realization of $K^1(X,\R/\Z).$

\item[Basu's description] is a bundle theoretic relative description  \cite{Ba}. It can be thought of as a realization of the suggestion of Atiyah, Patodi and Singer to describe $K^1(X,\R/\Z)$ with bundles of von Neumann algebras. Cycles are couples of finite rank vector bundles $E,F$ such that there exists a von Neumann bundle $V$  (relative to a semi-finite or $\rm{II}_{\infty}$ von Neumann algebra) such that $E\otimes V\simeq F\otimes V$. 
Addition of cycles is defined by direct sum and there is a natural equivalence relation. The equivalence classes form a group which is a realization of
 $K^1(X,\R/\Z)$.
\end{description}
\end{rem}


\subsection{Chern--Weil theory}
\label{CW}
Let $M$ be a closed manifold. Then the group $K_*(C(M), \R)$ as defined above is isomorphic to $K^{*}_{APS}(M, \R):=K^{*}(M)\otimes \R$ as defined in \cite{APS2}. This is in turn isomorphic to $H^*(M;\R)$
 using Chern--Weil theory.

\medskip 
Let $A$ be a  unital $C^*$-algebra. Let $W\longrightarrow M$ be a smooth bundle of finitely generated Hilbert $A$-modules over the manifold $M$: using the Serre--Swann theorem, this is a finitely generated projective module over $C^\infty(M;A)$, endowed with a nondegenerate $C^\infty(M)$ valued scalar product. We refer to 
\cite{MF} and \cite[Sections 2,3]{Sc} for the general theory of such bundles. 
A connection $\nabla$ is extended as usual to forms 
$$\nabla(\omega \otimes s)=d\omega \otimes s + (-1)^{\operatorname{deg}\omega}\omega \nabla s, \quad \omega \in \Omega^*(M), \quad s\in \Gamma(W).$$
The curvature $\nabla^2$ is a $2$-form with values endomorphisms. If the connection is metric, the curvature is skew--adjoint.

\subsubsection{Chern--Weil theory for $K_0$}
If $\Omega=\nabla^2\in \Omega^2(M;\operatorname{End}_A(W))$ is the curvature, the exponential 
$$
\operatorname{exp}\Big(\frac{\Omega}{2i\pi}\Big):=\sum_{k}\dfrac{\Omega \wedge \cdot \cdot \cdot \wedge \Omega}{(2i\pi)^kk!}\;\in \, \Omega^{2*}(M;\operatorname{End}_A(W)) \simeq  \Omega^{2*}(M)\otimes _{C^\infty(M)}\operatorname{End}_A(W)
$$ 
is well defined as a finite sum.

A trace $\tau:A\longrightarrow \C$ extends to $M_n(A)$ and therefore to $\operatorname{End}_A(E)\simeq pM_n(A)p$ for every finitely generated projective $A$-module $E\simeq pA^n$ (where $p\in M_n(A)$ is a projection).
If $W\longrightarrow M$ is a finitely generated projective  $A$-module  bundle, we thus obtain a well defined trace, which is a $C(M)$-linear map still denoted $\tau:\operatorname{End}_A(W)\to C(M)$. The Chern character of $W$ associated with $\tau$ is defined as follows. Put
$$
\Ch{}_{\tau}(\Omega):=(\id \otimes \tau)\Big(\operatorname{exp}\Big(\frac{\Omega}{2i\pi}\Big)\Big)\; \in \,\Omega^{2*}(M;\C).
$$
From the trace property it follows that the Chern form is closed. Its cohomology class does not depend on the connection. If $\tau $ is self adjoint, this class is seen to be real by taking $\nabla$ metric.

\medskip Hilbert $A$-module bundles define the $K$-theory of a compact space with coefficients in $A$. Indeed $K(M;A)$ is the $K$-group of the category of finitely generated projective Hilbert $A$-module bundles. The Serre--Swann functor sending a bundle to the module of its continuous functions establishes a canonical isomorphism
$$K^0(M;A)\simeq K_0(C(M;A))\simeq K_0 (C(M)\otimes A).$$ 
If  $A$ is a $\rm{II}_1$-factor, from the commutativity of the diagram
\begin{equation}\label{chDiag}
\xymatrix{K^0(M)\otimes K_0(A)\ar[d]_{\operatorname{Ch}\otimes \tau}\ar[r]^-{\simeq}&K^0(M;A)\simeq K_0(C(M)\otimes A)\ar[dl]^{\operatorname{Ch}_{\tau}}\\
H^{2*}(M;\R)
}
\end{equation}
one gets that $\operatorname{Ch}_{\tau}$ is an isomorphism.


\subsubsection{Chern--Weil theory for $K_1$}


The odd $K$-theory with coefficients in $A$ is defined by 
$$K^1(M;A):=K^0_c(M\times \R;A)\simeq K_0(C(M)\otimes C_0(\R)\otimes A)=K_0(C(M)\otimes SA)$$ where $SA=C_0(\R;A)$ is the suspension of $A$.

We then write $K^1(M;A)\times K^0(M;A)=K^0(M\times S^1;A)$. Also $H^{2*}(M\times S^1,\C)\simeq H^{*}(M,\C)$. 
The map $H^{2k}(M\times S^1,\C)\to H^{2k-1}(M,\C)$ is obtained by integration along the $S^1$ fibers. Using these identifications, if $\tau $ is a trace on $A$, we have a commuting diagram:
$$
\xymatrix{K^0(M\times S^1;A)\ar[r]^{\operatorname{Ch}_{\tau}}\ar[d]& H^{2*}(M\times S^1,\C)\ar[d]\\
K^1(M;A)\ar[r]^{\operatorname{Ch}^{od}_{\tau}}&H^{2*+1}(M;\C) }
$$
which defines the odd $\operatorname{Ch}_{\tau}^{od}$ in the bottom line. In particular, if $A$ is a $\rm{II}_1$-factor, then $\Ch_\tau$ is an isomorphism.

 Let $W\longrightarrow S^1\times M$ be a smooth bundle of finitely generated Hilbert $A$-modules over $S^1\times M$ and $\nabla$ a connection on $W$. Write $S^1=\R/\Z$ and use $\nabla$ in order to trivialize the bundle $W$ along $[0,1]$. We thus obtain a constant bundle $W_0$ over $M$ and a family  $(\nabla_u)_{u\in [0,1]}$ of connections on $W_0$. In other words, we may write (over $M\times [0,1]$), $$\nabla=(\nabla_u)+du\frac{\partial}{\partial u}\,\cdot$$ 
 Then $\nabla^2=\dot{\nabla}_u\wedge du+\nabla_u^2$, whence:
 $$\operatorname{Ch}_{\tau}^{od}(W,\nabla)=(2\pi i)^{-1}\int_0^1\tau \left(\dot{\nabla}_u\exp \Big({\frac{{\nabla}_u^2}{2\pi i}}\Big)\right)du\in   \Omega^{2*+1}(M;\C).$$

\medskip Alternatively, an element of $K^1(M;A)$ is represented by a continuous (smooth) map $\a:M\longrightarrow \operatorname{U}_n(A)$. The corresponding bundle over $M\times S^1$ is obtained by gluing the trivial bundle $M\times [0,1]\times A^n$ by means of $\a$.

One then may compute $\Ch_\t$ for this bundle and get (\cite[Prop. 1.2]{G})
\begin{equation}\label{oddCh}
\Ch{}_\t(\alpha)=\sum_{k=0}^\infty (-1)^{k}\frac{k!}{(2k+1)! (2\pi i)^{k+1}}(\id\otimes \t)(\alpha^{-1}d\alpha)^{2k+1}\in \Omega^{2*+1}(M;\C)\ .
\end{equation}

\subsubsection{Chern--Simons class}

Let $A$ a $C^*$-algebra and $\tau\colon A\to \C$ a continuous trace.
For a bundle $E$ of finitely generated projective Hilbert $A$-modules equipped with two connections $\nabla_0,\nabla_1$, the odd Chern--Simons differential form is 
\begin{equation}
\label{CSdef}
   \operatorname{CS}_{\tau}(\nabla_0,\nabla_1):=(2\pi i)^{-1}\int_0^1\tau \left(\dot{\nabla}_u\exp \Big({\frac{{\nabla}_u^2}{2\pi i}}\Big)\right)du\in   \Omega^{2*+1}(M;\C)
\end{equation}
where $\nabla_u:=(1-u) \nabla_0+ u \nabla_1$. Recall that $d(\operatorname{CS}_{\tau}(\nabla_0,\nabla_1))=\Ch_\t(\nabla_0)-\Ch_\t(\nabla_1)$.

\begin{dfn}\label{fbi}
A flat Hilbert $A$-module bundle is a pair $(W, \nabla)$ where $W$ is a Hilbert $A$-module bundle, and $\nabla $ is a flat connection. 

A \emph{flat bundle isomorphism} between  two flat Hilbert $A$-module bundles $(W, \nabla^{W})$ and $(V, \nabla^{V})$ is a bundle isomorphism $\psi\colon W\longrightarrow V$ which preserves the connection.
\end{dfn}

We fix the following notation, used repeatedly in the following. If  $(W, \nabla^{W})$ and $(V, \nabla^{V})$ are flat Hilbert $A$-module bundles and $\a\colon W\longrightarrow V$ is a isomorphism which does not necessarely preserve the connection, we use \eqref{CSdef} to define the closed form
\begin{equation}
 \operatorname{CS}_\tau(\a^{*}\nabla^{V}, \nabla^{W})\in \O^{odd}(M;\C)\ .
\end{equation}

From the above discussion we immediately get
\begin{prop}\label{Ch=CS}
If $\alpha:V\to V$  is an automorphism of a flat Hilbert $A$-module bundle, then $\alpha $ defines a class in $K_1(C(M)\otimes A)$ and the classes $\operatorname{Ch}_\t(\a)$ and $\operatorname{CS}_\tau(\a^*\nabla^{V}, \nabla^{V})$  in $H^*(M;\C)$ coincide. \hfill$\square$
\end{prop}


\bigskip
\section{Atiyah's theorem for covering spaces }

A key ingredient in our construction of the element $[\alpha]\in K^1(M;\R/\Z)$ is the fact that every flat bundle on $M$ with fibers a (finite) von Neumann algebra is trivial  (in $K$-theory). In fact this statement is equivalent to Atiyah's theorem on covering spaces \cite{Ati} (more precisely, to the versions of L\"uck and Schick in \cite{Lu,Sc}). See also remark \ref{Atiyah_l1} below for a further generalization.

Let $M$ be a closed manifold with  $\pi_1(M)=\G$. For an elliptic operator $D$ acting on the sections of a bundle $S\longrightarrow M$, let $\tilde D$ be the lift on the universal covering.   Recall that the index class $$\Ind \tilde D\in K_*(C^{*}\G)$$ is the $KK$-product $\left[\cV\right]\otimes_{C(M)}[D]$, where $\cV=\Mt\times_\G C^{*}\G$ is the so called Mishchenko bundle. Indeed, under the identification of the sections of $\mathcal{V}\otimes S$ with the $\Gamma$-invariant sections of the lifted bundle $\pi^*(\mathcal{V}\otimes S)=(C^*\G\otimes \pi^*(S))\times \tilde{M}$, one easily sees that $\tilde{D}$ is a connection and a $C^*\Gamma$-Fredholm operator. Its index is the Connes--Moscovici index class given by the idempotent $\begin{pmatrix}0&0\\0&I\end{pmatrix}+R$ for
$$R:=\left(\begin{array}{cc}S_0^2 & S_0(I+S_0)\tilde{Q} \\S_1 \tilde{D} & -S_1^2\end{array}\right)$$
where $\tilde{Q}$ is any almost local parametrix of $\tilde{D}$ and $S_0=I-\tilde{Q}\tilde{D}$, $S_1=I-\tilde{D}\tilde{Q}$ are smoothing.

Let now $B$ a $\rm{II}_1$-factor. A morphism $u \,\colon \, \G\to \cU(B)$ is a representation of $\G$ and therefore extends to a $*$-morphism $\bar u \,\colon \,C^{*}\G\to B$. 
Denote by $E_u$ the bundle $ E_u:=\tilde M\times_u B$ over $M$ with fiber $B$, let $\mathcal{E}_u$ denote the corresponding (finitely generated projective right) $(C(M)\otimes B)$-module of sections of $ E_u$, and $[\mathcal{E}_u]$ its class in $K_0(C(M)\otimes B)$. 

The von Neumann index of $D$ (in $K_0(B)$) is
\begin{equation}
\label{indB}
\Ind \tilde D_B=\bar u(\Ind \tilde D)=\left[ E_u\right]\otimes_{C(M)} [D]
\end{equation}

\begin{lem} 
  \label{Aty}
  $\mathcal E_u$ is the trivial element in $K_0(C(M)\otimes B)$.
\end{lem}

\begin{proof}
 The Chern character $\ch\,\colon\, K_0(C(M)\otimes B)\longrightarrow  H^{ev}(M, \R)$ is an isomorphism, as recalled in Section  \ref{CW}. Then the statement follows since $E_u$ is a flat bundle. 
\end{proof}

\begin{rem}
\begin{itemize} 
\item Atiyah's $L^2$-index theorem (L\"uck's and Schick version \cite{Lu,Sc}) follows immediately from Lemma \ref{Aty}. In fact, given any trace $\tau$ on $C^*(\Gamma)$, we may construct a ${\rm II}_1$-factor $B$ with a trace-preserving embedding $C^*(\Gamma)\to B$. It immediately follows from Lemma \ref{Aty} and formula  (\ref{indB}) that $\Ind \tilde D_B=\Ind D$. 
\item Conversely, Lemma \ref{Aty} follows from Atiyah's $L^2$-index theorem (L\"uck's and Schick version \cite{Lu,Sc}). Indeed, for any trace $\tau$ on $B$,
$\tau(\Ind_B \tilde{D})=\tau (\Ind D \cdot [1])$, and the result follows from the fact that the pairing map
$K^0(C(M))\otimes \R\longrightarrow \operatorname{Hom}(K^0(M),\mathbb{R})$, $D\longmapsto \langle \cdot, D\rangle$ is surjective since $K^0(M)$ is finitely generated.
\item It would be interesting to find a $K$-theoretic proof of Lemma \ref{Aty} with the idea:
 find a bigger ${\rm II}_1$-factor $B_1$, \emph{i.e.} a ${\rm II}_1$-factor containing $B$, such that the bundle $E_u\otimes _{C(M)\otimes B}C(M)\otimes B_1$ becomes explicitly trivial.
\end{itemize}
 \end{rem}

\begin{rem} (Atiyah's theorem for continuous traces on $\ell^1(\G)$).\label{Atiyah_l1} The symmetric index of an elliptic (pseudo)-differential operator $D$ is well defined as an element $\ind_{\ell^1(\G)}(\tilde D)$ of $K_0(\ell^1(\G))$ (see \cite{CM}). 

A continuous linear form on $\ell^1(\G)$  is of the form $a\mapsto \sum_{g\in \Gamma} a_g f(g)$ for some $f\in \ell^\infty(\G)$. The trace property holds if and only if $f$ is constant on conjugacy classes. Let $\tau$ be a trace associated with $f\in \ell^\infty(\G)$.

Now $\tau(\ind_{\ell^1(\G)}(\tilde D))=T_\t(R)$ where $T_\tau$ is defined as follows: let first $p:\tilde M\to M$ be the covering map. Note that $\Gamma $ acts freely and properly on $\tilde M$ and $M=\G \backslash \tilde M$ \emph{i.e.} the fibers of $p$ are the orbits of $\Gamma$.

Let $G=\Gamma\backslash \tilde M^2$ - where $\Gamma $ acts diagonally on $\tilde M\times \tilde M$. Then $G$ is a Lie groupoid and $C_c^\infty(G)$ is the algebra of $\G$-invariant smooth kernels on $\tilde M$ with a bounded support condition. Note that $R\in M_n(C_c^\infty(G))$.

Let $\varphi\in C_c^\infty(G)$. Let $x\in \tilde M$. Let $\varphi_x:\G\to \C$ be defined by $\varphi_x(g)=\varphi(gx,x)$. For $h\in \G$ and $g\in \G$, we have $\varphi(ghx,hx)=\varphi(h^{-1}ghx,x)$ and therefore $\tau(\varphi_x)=\tau(\varphi_{hx})$. We then put $T_\t(\varphi)=\int _M\tau(\varphi_x)\,dx$.

Now, if $\tilde D$ and $\tilde Q$ are taken local enough, we have $R(x,gx)=0$ for all $x\in \tilde M$ and $g\ne 1$, and therefore $$\tau ({\rm ind}_{\ell^1(\G)}(\tilde D))=\t(1)\,\varepsilon({\rm ind}_{\ell^1(\G)}(\tilde D))=\tau(1)\, \ind (D)$$ where $\varepsilon:\G\to \C$ is the trivial representation of $\G$.
\end{rem}


\section{$[\a]_{new}$ in $K^1(M;\R/\Z)$ associated with a flat bundle $\a\colon \G\to U(n)$}

\label{section5}

The triviality property of Lemma \ref{Aty} is crucial in what will be our definition of the element  in $ K^1(M;\R/\Z)$ associated to flat a vector bundle.

Let $M=\G\backslash \tilde M$ be a closed manifold with $\pi_1(M)=\G$ and universal cover $\tilde M$. A flat bundle over $M$ is just given by a homomorphism from $\G$ to the associated structure group.

Let $B$ be a (finite) von Neumann algebra. A flat $C(M)\otimes B$ module with fiber $B^n$ is given by a homomorphism $u$ from $\G$ to the group $U_n(B)$ of unitaries of $M_n(B)$. The associated bundle is $\tilde M\times_\Gamma B^n$; the corresponding finitely generated projective module over $C(M)\otimes B$ is the set of continuous maps $f:\tilde M\to B^n$ such that $f(g x)=u(g)f(x)$ for all $g\in \G$ and $x\in \tilde M$. A flat isomorphism of flat bundles associated with $u_+:\G\to U_n(B)$ and $u_-:\Gamma \to U_n(B)$ is given by a unitary $v\in U_n(B)$ intertwining the morphisms $u_+$ and $u_-$.

We first observe the following 

\begin{prop}
   \label{reduct} 
Let $M$ be a compact manifold and $B$ be a finite factor. Let $\cE $ be a flat $C(M)\otimes B$ module with fiber $B$  (given by a homomorphism $u:\pi_1(M)\to U(B)$). Then there exists $\ell \in \mathbb{N}^*$ such that  $ \cE ^\ell\simeq (C(M)\otimes B)^\ell$.
\end{prop}

\begin{proof}
Follows immediately from Lemma \ref{Aty} using Lemma \ref{full2}.
 \end{proof}

A flat Hermitian vector bundle $E_\a\to M$ is given by a homomorphism $\a\colon \G\to U_n$.

In the following we denote with $E^+=C(M, E_\a)$ the corresponding (flat) $C(M)$-module of sections and $E^-=C(M)^n$  the module of sections of the trivial bundle of rank $n$.

\begin{prop} 
   \label{prop.RmodZ}
 Let $E_\a$ be a flat unitary vector bundle as above.
  \begin{itemize}
   \item[a)] There exists a finite factor $B$, a flat $C(M)\otimes B$ module  $\cE $ with fiber $B$  and an isomorphism of flat  
   bundles $v\colon E^-\otimes_{C(M)} \cE \longrightarrow E^+\otimes_{C(M)} \cE $.
     \end{itemize}
By Proposition \ref{reduct}, up to tensoring $B$ by  $M_n$, we  have an isomorphism $\phi\colon C(M)\otimes B\to \cE$. 
   Let $w_v$ be the isomorphism  
   \begin{equation}
\label{iso}
   w_v:=(1_{E^-}\otimes \phi)^{-1}\circ v^{-1}\circ (1_{E^+}\otimes \phi)\colon E^+\otimes B\to E^-\otimes B \ . 
\end{equation}
   \begin{itemize}
  \item[b)] 
   The class of $(E^+,E^-, w_v)$  in $K^1(M, \R/\Z)$ is independent of all choices $B$, $\cE $, $v$, $\phi$.
  \end{itemize}
\end{prop}

\begin{proof}
\begin{itemize}
\item[a)] Given a finite factor $B$ and a flat $C(M)\otimes B$ module  $\cE $ with fiber $B$ associated with a morphism $u:\G\to U(B)$, the flat bundle $E^-\otimes _{C(M)}\cE$ corresponds to the morphism $g\mapsto 1\otimes u(g)\in U_n(B)\subset M_n\otimes B$ and the flat bundle $E^+\otimes _{C(M)}\cE$ corresponds to the morphism $g\mapsto \alpha(g)\otimes u(g)$.

We therefore just have to construct a finite  factor $B$ with a morphism $u:\G\to U(B)$ and an element $v\in U_n(B)$,  such that, for all $g\in \G$, we have $v(1\otimes u(g))=(\a(g)\otimes u(g))v$.

Let $\G_1=\a(\G)$ be the image of $\G$ in $U_n$ (with the discrete topology) and $K=\overline{\G_1}$ its closure in $U_n$. Set $B_1=L^\infty(K)\rtimes \G_1$, where $\G_1$ acts on $K$ by left translation. Denote by $u:\G_1\to L^\infty(K)\rtimes \G_1$ the canonical inclusion; the inclusion map $v:K\to U_n\subset M_n$ is an element of $M_n(L^\infty(K))\subset M_n(B_1)$; by definition $(1\otimes u_g)v(1\otimes u_g)^{-1}$ is the function $x\mapsto v(g^{-1}x)=\alpha(g)^{-1}v(x)$, \emph{i.e.} the element $(\a(g)^{-1}\otimes 1)v$ in $M_n\otimes L^\infty(K)\subset M_n(B_1)$. 

The desired equality $v(1\otimes u(g))=(\a(g)\otimes u(g))v$ follows.

By density of $\G_1$ (discrete) in $K$ (compact), we have that $B_1$ is a finite factor.

%

\item[b)] The choice of $\cE$ is equivalent to the choice of a morphism $u:\G\to U(B)$; the choice of $\varphi$ is equivalent to a  continuous map $\psi:\tilde M\to U(B)$ such that $\psi (g x)=u(g)\psi (x)$ for all $g\in G$ and $x\in \tilde M$. The element $w_v^{-1}$ is then the map $x\mapsto (1_{M_n}\otimes \psi (x))^{-1}v(1_{M_n}\otimes \psi (x))$ from $\tilde M$ to $U_n(B)$ which satisfies $\psi (gx)=(\alpha (g)\otimes 1)\psi (x)$ and thus defines an isomorphism of bundles $E^-\otimes B\to E^+\otimes B$. 

Assume that we are given $(B_j,u_j,v_j,\psi_j)$ with $j=1,2$. Put $B=B_1\overline\otimes B_2$ (von Neumann tensor product of course),  and define $u=u_1\otimes u_2:g\mapsto u_1(g)\otimes u_2(g)$ and $\psi=\psi_1\otimes \psi_2:x\mapsto \psi_1(x)\otimes \psi_2(x)$; put also $V_1=v_1\otimes 1_{B_2}$ and $V_2=\beta(v_2)$, where $\beta:M_n\otimes B_2\to M_n\otimes B_1\otimes B_2$ is the obvious inclusion.

We find in this way two new families $(B,u,V_1,\psi)$ and $(B,u,V_2,\psi)$ with the desired properties.

For $x\in \tilde M$, we have 
$w_{V_1}(x)=(1_{M_n}\otimes \psi (x))^{-1}V_1^{-1}(1_{M_n} \otimes \psi (x))=w_{v_1}(x)\otimes B_2$; 
it follows that the elements of $K_1(M,\R/\Z)$ corresponding to $(B_1,u_1,v_1,\psi_1)$ and $(B,u,V_1,\psi)$ coincide. Also $w_{V_2}(x)=\beta (w_{v_2}(x)$, and thus the elements of $K_1(M,\R/\Z)$ corresponding to $(B_2,u_2,v_2,\psi_2)$ and $(B,u,V_2,\psi)$ coincide.

Finally, write $V_1=V_2W$ where $W$ is a unitary in $M_n(Q)$ where $Q=B\cap \{u(g);\ g\in \G\}'$. By connectedness of the unitary group of a von Neumann algebra, it follows that  the elements of $K_1(M,\R/\Z)$ corresponding to $(B,u,V_1,\psi)$ and $(B,u,V_2,\psi)$ coincide.
\qedhere
\end{itemize}
\end{proof}

\begin{dfn}
    \label{alpha.new}
  Given $\a\colon \G\to U_n$, $E^+$ and $E^-$ as above. For any $\mathcal E, \phi, v$ as in Proposition \ref{prop.RmodZ}. We define 
$$
[\a]_{new}:=[(E^+, E^-, w_v)] \in K_1(C(M), \R/\Z)\ .
$$
\end{dfn}

\begin{theorem}
     \label{element.RmodZ}
      Let $E_\a$ be a flat unitary vector bundle as above.
Consider any other given data of: a finite factor $B$, a flat $B$-bundle $\mathcal F\to M$ along with an isomorphism $\theta\colon E^-\otimes \mathcal F\to    E^+\otimes \mathcal F$ (not required to preserve   connections),  and a  trivialization  $\psi\colon \mathcal F\to C(M)\otimes B$, and denote with $\CS(\theta^*\nabla_{E^+\otimes \mathcal F};\nabla_{E^-\otimes \mathcal F})$ the Chern--Simons transgression form (formula \ref{CSdef}). Let $w_\theta= (1_{E^+}\otimes \psi)^{-1}\circ \theta\circ (1_{E^-}\otimes \psi):E^-\otimes B\to E^+\otimes B $. We have
$$
   [(E^+, E^-, w_\theta^{-1})] +j_*\left(\Ch{}_\t ^{-1}(\CS(\theta^*\nabla_{E^-\otimes \mathcal F};\nabla_{E^+\otimes \mathcal F}))\right)=[\alpha]_{new}\in K^1(M, \R/\Z).
$$
\end{theorem}

\begin{proof}
Up to tensoring $\cF$ by $\cE$ given by proposition \ref{prop.RmodZ}, we may assume that there is a flat connexion preserving isomorphism $v\colon E^+\otimes \mathcal F\to    E^-\otimes \mathcal F$.

Let now $B$ and $\mathcal F\to M$ be fixed, and $\theta\colon E^-\otimes \mathcal F\to E^+\otimes \mathcal F$ and $v\colon E^+\otimes \mathcal F\to    E^-\otimes \mathcal F$ two isomorphisms, such that $v$ is connexion preserving. 

Put $\ell=v\circ \theta$; it is an automorphism of $E^-\otimes \mathcal F$. As $v^*(\nabla_{E^-\otimes \mathcal F})=\nabla_{E^+\otimes \mathcal F}$, we find (using prop. \ref{Ch=CS})
$$\CS(\theta^*\nabla_{E^+\otimes \mathcal F};\nabla_{E^-\otimes \mathcal F})=\CS(\ell^*\nabla_{E^-\otimes \mathcal F};\nabla_{E^-\otimes \mathcal F})=\Ch(\ell).$$

Moreover $[(E^+, E^-, w_{ v})]-[(E^+, E^-, w_{\theta}^{-1})]=[(E^-, E^-,  w_{v}\circ w_{\theta}]$ is the image by $j:K_1(C(M);\R)\to K_1(C(M);\R/\Z)$ of the class of the automorphism $w_{v}\circ w_{\theta}=(1_-\otimes \phi)^{-1}\circ \ell\circ (1_-\otimes \phi)$ of $E^-\otimes B$. This is the same class as the automorphism $\ell$ of $E^-\otimes \mathcal F$.
\end{proof}

The element $[\a]_{APS}$ is constructed in \cite[5]{APS3} by means of an isomorphism $\theta:E^-\otimes B\to E^+\otimes B$, where $B=M_n(\C)$ (for a suitable $n$): it is defined by $[\a]_{APS}=[(E^+, E^-, w_\theta^{-1})] +j_*\left(\Ch{}_\t ^{-1}(\CS(\theta^*\nabla_{E^-\otimes B};\nabla_{E^+\otimes B}))\right)$. We immediately find: 
 
\begin{cor} 
Let $\a\colon \G\to U(n)$.   The element $[\a]_{new}$ coincides with $[\a]_{APS}$. \hfill$\square$
\end{cor}


 \section{The pairing with $[D]\in K_1(M)$ and the spectral flow description}\label{pairing-section}
 
Let $[\a]=[(E^+, E^-, w_v)] \in K_0(C_{i_0}\otimes C(M))$ of Section \ref{section5}, where $C_{i_0}$ is the mapping cone of the unital inclusion $i_0:\C\to B$ - here $B$ is a ${\rm II}_1$ factor. In order to stay within unital algebras, we represent $[\a]$ in the $K$-theory of the double cylinder $Z_{i_0}\otimes C(M)$, as we saw in \eqref{split}. It is the image in the summand $K_0(C_{i_0}\otimes C(M))$ of the $Z_{i_0}\otimes C(M)$-module
 \begin{equation*}
 F_{w_v}=\{ f\in C\left([0,1],E^+\otimes B \right); 
f(0)\in E^+\otimes 1 \hbox{ and }\ w_v f(1)\in E^-\otimes 1\}\ .
 \end{equation*}

The Kasparov product gives the natural pairing between the class $[\a]\in K_0(C_{i_0}\otimes C(M))$ and a $K$-homology class $[D]\in K_1(M)$. In this section we describe and interpret it. In other words, we interpret the map $K_1(M)=KK^1(C(M),\C)\to \R/\Z$ induced by $[\a]$ via the pairing
\begin{equation}\label{KKprod}
  KK^0(\C, C(M)\otimes Z_{i_0})\times KK^1(C(M),\C)\longrightarrow KK^1(\C, Z_{i_0})=K_1(Z_{i_0})=  \R/\Z\ .
\end{equation}

\subsection{Various notions of connections}
Let us start with a class $[D]\in K_1(M)$ represented by the unbounded Kasparov module $(H, D)$, where $D$ is a first order self-adjoint elliptic (classical) pseudodifferential operator acting on the sections of a Hermitian vector bundle $S\to M$, and $H=L^2(M,S)$. 

A particularly important case for $D$ in our discussion is the case of a Dirac type operator, \emph{i.e.} an operator constructed in the following way: \begin{itemize}
\item We assume that $M$ is endowed with a Riemannian metric, 
\item The bundle $S$ is a \emph{Clifford bundle} \emph{i.e.} it is endowed with a linear bundle map $\cliff:T^*M\to \End(S)$ such that for every $\xi\in T^*_xM$ then $\cliff(\xi)\in \End(S_x)$ is a skew-adjoint and satisfies $\cliff(\xi)^2=-\|\xi\|^21_{S_x}$; denote by  $c:S\otimes T^*M\to S$ the corresponding map.
\item The bundle $S$ is further endowed with  a \emph{metric Clifford connection} $\nabla_S$  in the sense of \cite[Def. 3.39]{BGV}.
\item The corresponding Dirac operator is given by $D(s)=c(\nabla_S (s))$ for any smooth section $s$ of $S$.

\end{itemize}

Let us distinguish different types of connections that arise in our discussion:

\begin{dfn} Let $A$ be a unital $C^*$-algebra $E$ a finitely generated projective module endowed with a compatible Hilbert $A$-module structure. Let $V\to M$ be a smooth bundle with fibers $E$. 
\begin{itemize}
\item If $D$ is a first order self-adjoint elliptic (classical) pseudodifferential operator acting on the sections of a hermitian vector bundle $S\to M$, a \emph{pseudodifferential connection} for $D$ is  a selfadjoint pseudodifferential $A$ operator $D_V$  (in the sense of \cite{MF}) acting on the sections of the bundle $V\otimes S$ whose principal symbol is $\sigma_{D_V}=1_V\otimes \sigma_D$.
\item If $D$ is Dirac type, then a \emph{Dirac type connection} for $D$ is an operator of the form $(1_V\otimes c)\circ \nabla_{V\otimes S}$ where  $\nabla_{V\otimes S}$ is a metric Clifford connection on the Clifford $A$-bundle ${V\otimes S}$.
\end{itemize}
\end{dfn}

Of course, if $D$ is Dirac type, then a Dirac type connection for $D$ is a particular case of a pseudodifferential connection.

\subsection{The pairing as an element of $K_1(Z_{i_0})$}
The product \eqref{KKprod} is computed by any choice of a $D$-connection $G$ in the sense of  \cite{CS, Ku} on the Hilbert $Z_{i_0}$-module $V= F_{w_v}\otimes_{C(M)\otimes Z_{i_0}} \left(L^{2}(M,S)\otimes Z_{i_0}\right)$. Note that a pseudodifferential connection for $D$ is a $D$-connection in the sense of \cite{CS, Ku} (the converse is of course not true).

We have
\begin{equation}\label{pairing}
    [\a]\otimes_{C(M)} [D]=[(V,  G)]\ .
\end{equation} 

We can take any $G$ as follows. It will be of the form $G=(G_t)$, \emph{i.e.} $G(f)(t)=G_t(f_t)$. Since the bundle $E^-=C^n\otimes M$ is trivial, we may set $G_1=w_v(1_{\C^n}\otimes D\otimes 1_B)w_v^{-1}$. To construct $G_0$, we may just take any selfadjoint pseudodifferential operator $D_{E^+}$ on $E^+\otimes S$ whose principal symbol is $1_{E^+}\otimes \sigma_D$; we may also assume that $D$ is local enough (in particular if $D$ is differential), so that there is a canonical choice for $D_{E^+}$, called the \emph{twisted tensor product} of $D$ with the flat bundle $E^+$.

Put then $G_0=D_{E^+}\otimes 1_B$. The operators $G_0$ and $G_1$ have the same principal symbol and therefore $G_1-G_0$ is bounded. We may then choose any strongly continuous bounded path $q_t$ of selfadjoint operators acting on $E^+\otimes S\otimes B$ such that $q_0=0$, $q_1=G_1-G_0$ and set $G_t=G_0+q_t$. In particular, we may define $G_t=(1-t)G_0+t G_1$.

\subsection{The pairing as an element of $\R/\Z$}

Let us first recall a few facts about index in the type ${\rm II}$ setting:

Let $B$ be a factor of type ${\rm II}_1$. Denote by $\tau $ the normalized trace on $B$. Let $N=B\overline \otimes \mathcal{L}(H)$ the associated factor of type ${\rm II}_\infty$ endowed with the trace  ${\rm Tr}_N=\tau\otimes {\rm Tr}$ where ${\rm Tr}$ is the canonical trace on $\mathcal{L}(H)$. Let $p=1\otimes p_0\in N$, where $p_0$ is a minimal projection in $\mathcal{L}(H)$ and identify $B$ with $pNp$.

Let $P_\pm \in N$ be two projections such that $P_+-P_-$ is in the ideal $\cK_N$ of generalized compact operators in $N$. Then $\ind(P_+,P_-)$ is the index of $q:E^+\to E^-$ where $E_\pm=P_\pm Np$ considered a a Hilbert $B$-module and $qx=P_-x$ for $x\in E^+$. Note that $q$ is $B$-Fredholm with quasi-inverse $x\mapsto P_+x$. If $P_+-P_-$ is in the domain of ${\rm Tr}_N$, then $\ind (P_+,P_-)={\rm Tr}_N(P_+-P_-)$.

Let $D_\pm$ be selfadjoint operators affiliated with $N$ such that $D_+-D_-\in N$ (which means that $D_+$ and $D_-$ have same domain and there exists $b\in N$, with $b=b^*$ and $D_+=D_-+b$). Note that for $\lambda_\pm$ not in the spectrum of $D_\pm$, we have $$(D_--\lambda_-)^{-1}-(D_+-\lambda_+)^{-1}=(D_+-\lambda_+ )^{-1}(b+\lambda_--\lambda_+)(D_--\lambda_-)^{-1},$$ therefore the resolvent of $D_+$ is in $\cK_N$ if and only if the resolvent of $D_-$ is. Assume this is the case. For $a\in \R$, let $\chi_a$ be defined by $$\chi_a(t)=1 \ \ \hbox{if}\ \ t\ge a\ \ \hbox{and}\ \  \chi_a(t)=0 \ \ \hbox{if}\ \ t< a.$$ Then for all $a\in \R$, we have $\chi_a(D_+)-\chi_a(D_-)\in \cK_N$ and we put $$\SF_a(D_+,D_-)=\ind(\chi_a(D_+),\chi_a(D_+)).$$ Note that for all $a,b\in \R$, we have (by additivity of the index) $$\SF_a(D_+,D_-)-\SF_b(D_+,D_-)={\rm Tr}_N((\chi_a-\chi_b)(D_+))-{\rm Tr}_N((\chi_a-\chi_b)(D_-)).$$

Let us now come to $(V,G)$ as defined above. We easily find:

\begin{prop}
The class $j(\SF_a(G_0,G_1))\in \R/\Z$ is independent of $a\in \R$. The class $[\a]\otimes_{C(M)} [D]$ in $\R/\Z$ is $j(\SF_a(G_0,G_1))$. 
\end{prop}

\begin{proof}
Since $G_0=D_{E^+}\otimes 1_B$ and $G_1=w_v(1_{\C^n}\otimes D\otimes 1_B)w_v^{-1}$, the spectra of $G_0$ and $G_1$ are discrete and the trace of all the spectral projections is integer. In other words, ${\rm Tr}_N((\chi_a-\chi_b)(G_0))\in \Z $ and ${\rm Tr}_N((\chi_a-\chi_b)(G_1))\in \Z$, and therefore $j(\SF_a(G_0,G_1))\in \R/\Z$ is independent of $a\in \R$. 

To see that the class $[(V,G)]$ in $\R/\Z$ is equal to $j(\SF_a(G_0,G_1))$, note that:\begin{enumerate}
\item there is a smooth function $\varphi:\R\to [-1+1]$ such that $\varphi=-1$ near $-\infty$, $\varphi=+1$ near $+\infty$, and $\varphi=\chi_a$ in $\Sp(G_0)\cup \Sp(G_1)$. The class of the unbounded Kasparov module $(V,G)$ is that of $(V,\varphi(G))$. 
\item \label{degener}If $\SF_a (G_0,G_1)=0$, one can find a path $P=(P_t)$ of projections in $\mathcal{L}(V)$ such that $t\mapsto P_t-P_0$ is a norm continuous path with values in $\cK((L^2(M;S\otimes E^+)\otimes B)$, and such that $P_0=\chi_a(G_0)$ and $P_1=\chi_a(G_1)$: the class of $(V,\varphi(G))$ coincides with the class of  $(V,2P-1)$ which is degenerate.
\item For $\lambda\in\R_+^*$, let $E_\lambda$ be a finitely generated projective $B$-module with trace $\lambda$ and denote $V_\lambda=C_0(]0,1[,E_\lambda)$. Let $F_\lambda \in \mathcal{L}(V_\lambda)$ be given by $(F_\lambda\xi)(t)=(2t-1)\xi(t)$. Then $(V_\lambda,F_\lambda)$ is the product of the (opposite of the) Bott element in $K_1(C_0(]0,1[)$ by $E_\lambda $,  and its class in $K_1(C_0(]0,1[)\otimes B)=\R$ is therefore equal to $-\lambda$.
\item If $\SF_a (G_0,G_1)=\lambda>0$, then the class of $(V,\varphi(G))\oplus (V_\lambda,F_\lambda)$ is $0$ by (\ref{degener}), hence the class of $(V,\varphi(G))$ is $j(\lambda)$. This is also true if $\SF_a (G_0,G_1)<0$ (for instance by replacing $G$ by $-G$). \qedhere
\end{enumerate}
\end{proof}

\subsection{The pairing as a rho invariant}
We will now explain how our approach can be used to establish the formula of Atiyah--Patodi--Singer in \cite[Sec.5]{APS3}.

We first recall the facts about the $\eta$ invariant that will be used here.  A very nice survey can be found in the first chapter and in appendix D of \cite{Bo}.
\begin{enumerate}
\item If $P$ is an elliptic self-adjoint pseudodifferential operator of positive order, the \emph{eta function} $\eta(P,.)={\rm Tr}(P|P|^{-1-s})$ has a meromorphic continuation which is regular at $0$ (cf. \cite[Thm. 4.5]{APS3}, \cite{Gi2}). Put $\xi(P)=\frac{1}{2}\Big(\eta (P)+\dim \ker P\Big)$
\item If $P$ is differential of Dirac type, then $\displaystyle\eta (P)=\frac 1{\sqrt \pi}\int _0^{+\infty} {\rm Tr}(Pe^{-tP^2})\frac {dt}{\sqrt t}$ (cf \cite[Thm 2.6]{BF}).
\item The above remain true for $P$ with coefficients in a finite von Neumann algebra $B$ (\cite{CG, Ram, CP1, LP, An}).
\item For $P$ with coefficients in $B$, put $\displaystyle\xi_\varepsilon (P)=\frac 1{2\sqrt \pi}\int _\varepsilon^{+\infty} {\rm Tr}(Pe^{-tP^2})\frac {dt}{\sqrt t}+\frac 12\dim{}_B  \ker P.$ If $P_1,P_0$ are bounded perturbations of each other, then for any $\ep>0$,  and any smooth path $P_t$ joining $P_0$ with $P_1$ (\emph{i.e.} a path of the form $t\mapsto P_0+Q_t$ where $t\mapsto Q_t$ is smooth) we have (\cite[Cor. 8.10]{CP1})
$$\SF _0(P_0,P_1)=\xi_\varepsilon (P_0)-\xi_\varepsilon (P_1)+\sqrt {\frac{\varepsilon}{\pi}} \int_0^1{\rm Tr}_N(\dot  P_t e^{-\varepsilon P_t^2})\, dt.
$$
(where $N$ is the ${\rm II}_\infty$ factor associated with $B$).
\item \label{locality} If $V\to M$ is a von Neumann bundle and $P_t$ is a smooth family of Dirac type operators then $\displaystyle\sqrt {\frac{\varepsilon}{\pi}} \int_0^1{\rm Tr}_N(\dot  P_t e^{-\varepsilon P_t^2}) \,dt$ converges, when $\varepsilon \to 0$, to a local term, $$\displaystyle\int_M \hat A(M)\ch S/\hat S \cdot\CS(\nabla_0, \nabla_1)$$   see \cite[Sec. 4]{APS2}, \cite[1.5.1]{Bo}.
\end{enumerate}

Let $D$ be an elliptic differential operator. Recall that the rho invariant of $\alpha$ is $\rho(\alpha,D)=\xi(D^\alpha)-n\xi(D)$.

\begin{prop} 
{\rm \cite[Sec.5]{APS3}:}  If $D$ is a Dirac type operator, then
$[\a]\otimes_{C(M)} [D]$ is the class in modulo $\Z$ of $\rho(\alpha,D)$.
\end{prop}
\begin{proof} If $P$ and $P'$ are two Dirac type operators acting on a von Neumann Clifford bundle, we put $\theta(P,P')=\SF _0(P,P')-\xi  (P)+\xi  (P')=\lim_{\varepsilon \to 0}\displaystyle\sqrt {\frac{\varepsilon}{\pi}} \int_0^1{\rm Tr}(\dot  P_t e^{-\varepsilon P_t^2}) \,dt$ for any smooth path $P_t$ joining $P$ with $P'$.

On the bundle $E^+ \otimes B\simeq_{1_{E^+}\otimes \phi} E^+\otimes \cE\simeq _v E^+\otimes \cE\simeq_{1_{E^-}\otimes \phi^{-1}}E^- \otimes B$ over $M$, we have three flat bundle structures (since the two middle ones coincide). We thus get three Dirac type operators as twisted tensor products of $D$ with this flat bundle $G_0=D^\alpha\otimes 1_B,\ \tilde G$ and $G_1=w_v(1_{\C^n}\otimes D\otimes 1_B)w_v^{-1}$ acting on $E^+\otimes S\otimes B$.

Consider the Dirac operators $P=D\otimes 1_B$ and $P'=(1_{S}\otimes \phi)\tilde D(1_{S}\otimes \phi^{-1})$ acting on $S\otimes B$.

We have $G_1=w_v(1_{\C^n}\otimes P)w_v^{-1}$ and $\tilde G=w_v(1_{\C^n}\otimes P')w_v^{-1}$, therefore $\theta(G_1,\tilde G)=n\theta(P,P')$.

By the locality property (\ref{locality}) of $\theta$, and since $G_0$ is the twisted tensor product of $P$ with the flat bundle $E^+$ and $\tilde G$ is the twisted tensor product of $P'$ with the flat bundle $E^+$, we find $\theta(G_0,\tilde G)=n\theta(P,P')$. And finally $\theta(G_0,G_1)=0$.
\end{proof}


\providecommand{\bysame}{\leavevmode\hbox to3em{\hrulefill}\thinspace}
\providecommand{\MR}{\relax\ifhmode\unskip\space\fi MR }
\providecommand{\MRhref}[2]{%
  \url{http://www.ams.org/mathscinet-getitem?mr=#1}{#2}
}
\providecommand{\href}[2]{#2}

\end{document}